\documentclass{scrartcl}

\usepackage{amsmath,amsthm,amsfonts,amssymb}
\usepackage{mathtools}
\usepackage[capitalise,noabbrev]{cleveref}

\newtheorem*{theorem*}{Theorem}
\newtheorem{lemma}{Lemma}

\newtheorem*{corollary*}{Corollary}

\newtheorem*{proposition*}{Proposition}

\newtheorem*{remark*}{Remark}

\newtheorem*{definition*}{Definition}

\title{On the definition of Suzuki\\ groups over rings}
\author{Andrei Smolensky\thanks{\,\ email: \texttt{andrei.smolensky@gmail.com}\newline Department of Mathematics and Mechanics, Saint Petersburg State University\newline The research was supported by RSF (project No. 17-11-01261)}}
\date{\today}

\begin{document}
\maketitle
\begin{abstract}
The definition of Suzuki groups over rings is given by means of an explicit description as a difference-algebraic group. For a (not necessarily perfect) field with more than two elements this construction produces a simple group.
\end{abstract}

Let $R$ be a commutative ring of characteristic $2$ posessing a Tits' endomorphism $\tau$, that is, and endomorphism that squares to the Frobenius endomorphism. An example of such a ring is the finite field with $2^{2n+1}$ elements, where $\tau(x)=x^{2^n}$.

Suzuki group over a field $F$ is defined as the subgroup of the symplectic group $Sp(4, F)$, pointwise invariant under the action of the exceptional automorphism, which is defined by its action on the elementary generators in terms of $\tau$ and the length-changing symmetry of the Dynkin diagram $\mathsf{C}_2$. This definition only works on the level of the elementary subgroup, and thus can not be directly generalized to the groups over rings.

Another definition treats Suzuki group as the subgroup of projective symplectic group, consisting of all projective mappings that commute with a certain polarity between the points in $\mathbb{P}^3$ and the lines of a line complex, defined by a linear equation on the Pl\"ucker coordinates. Over a field $F$ of characteristic $2$ there is no difference between the simply-connected and the adjoint groups of type $\mathsf{C}_\ell$, since the center of $Sp(4, F)$ is trivial. Over a ring $R$, however, this center might be larger, since it contains $\mu_2(R)$. This presents another obstacle to the definition of Suzuki group over rings.

The aim of the present note is to show that Tits' polarity map construction can be carried over to the vector space underlying $\mathbb{P}^3$, and that the resulting definition of Suzuki groups is easily extended to arbitrary commutative rings with Tits' endomorphism.

An explicit description was given by A.~Duncan in \cite{DuncanSp4}, but the construction there involves taking square roots, and most of the proofs assume the base ring to be a finite field. It is, moreover, not compliant with the usual matrix representation of $Sp(4, F)$. Our construction is consistent with the construction in Carter's book \cite{CarterLie} and the \texttt{GAP} realisation.

To give some motivation to the final definition, we give below a version of Tits' construction \cite{TitsSuzukiRee,TitsOvoides}, with some small details filled in.

Let $V$ be a vector space with basis $e^1, e^2, e^{-2}, e^{-1}$. Let $L$ be a line complex, that is, a set of lines in $\mathbb{P}V$, defined by the equation $p_{1,-1}=p_{2,-2}$ on their Pl\"ucker coordinates. This line complex is represented in $\mathbb{P}(\wedge^2V)$ by a subvariety $Q$ of the Klein quadric $K=\{ p_{1,2}p_{-2,-1} + p_{1,-2}p_{2,-1} + p_{1,-1}p_{2,-2} = 0 \}$.

Consider now the spaces tangent to $Q$. A hyperplane tangent to $V$ at $(q_{ij})_{i,j=1}^{-1}$ is defined by the equation (see \cite[Lecture~14]{HarrisAlgGeom})
\[ \sum q_{ij}p_{kl}=0,\ \text{the sum is over all}\ (\{i,j\},\{k,l\})\ \text{with all four indices distinct.} \]
The space tangent to $Q$ is the intersection of the hyperplane described above and the hyperplane tangent to $\{ p_{1,-1}=p_{2,-2} \}$, which is already a hyperplane. This intersection contains the point $O$ having $p_{1,-1}=p_{2,-2}=1$ and all other coordinates zero. On the other hand, the intersection of the hyperplanes tangent to $K$ at the four points with the sole non-zero coordinates respectively $q_{1,2}$, $q_{1,-2}$, $q_{2,-1}$ or $q_{-2,-1}$ contains no point with any of $p_{ij}$ other then $p_{1,-1}$ or $p_{2,-2}$ non-zero. Thus the intersection of the tangent spaces equals $\{ O \}$.

One can now define a projection from $O$ onto the $\{ p_{1,-1}=p_{2,-2}=0 \}\cong\mathbb{P}^3$ by
\[ (p_{1,2}:p_{1,-2}:p_{1,-1}:p_{2,-2}:p_{2,-1}:p_{-2,-1})\mapsto (p_{1,2}:p_{1,-2}:p_{2,-1}:p_{-2,-1}). \]

To a set of lines from $L$, passing through a point $(x_1:x_2:x_{-2}:x_{-1})$, there corresponds a line in $Q$, and thus, via the projection described above, a line in $\mathbb{P}^3$ with Pl\"ucker coordinates
\begin{align*}
& p_{1,-1}=p_{2,-2}=x_1x_{-1}+x_2x_{-2},\\
& p_{1,2}=x_1^2,\quad p_{1,-2}=x_2^2,\quad p_{2,-1}=x_{-2}^2,\quad p_{-2,-1}=x_{-1}^2.
\end{align*}
Thus we have established two mappings $\delta\colon\mathbb{P}V\to L$ and $\rho\colon L\to\mathbb{P}V$. If one assumes $F$ to be perfect, these two maps are bijections. This allows to define on the group $G(\mathbb{P}V, L)$ of all projectivities of $\mathbb{P}V$ conserving $L$ a map $\delta^*$ such that for any $g\in G(\mathbb{P}V, L)$ and $d\in L$ one has $\delta^*(g)(\rho(d))=\rho(g(d))$ (the right-hand side is defined since $g(d)\in L$). The action of an operator $g$ on $L\subset\mathbb{P}(\wedge^2V)$ can be described by $g\wedge g$, so $\delta^*(g)=\rho\circ(g\wedge g)\circ\rho^{-1}$.

On the other hand, the Tits automorphism $\tau$ on $\mathbb{P}V$ induces a map $\tau^*$ on $G(\mathbb{P}V, L)$ by $\tau^*(g)=\tau g \tau^{-1}$. The Suzuki group is then defined as the group of all $g\in G(\mathbb{P}V, L)$ such that $\delta^*(g)=\tau^*(g)$.

If the Frobenius endomorphism is not invertible (and hence so is $\tau$), the map $\tau^*$ is not defined, and so this definiton has to be changed. Since we are working over a ring, we will also lift the above mappings to the level of the underlying vector spaces (or free modules) and simply-connected symplectic group (and denote them by the same letters).
\begin{alignat*}{2}
\rho\colon \wedge^2V\to V,\quad & p_{1,2}e^1\wedge e^2 + p_{1,-2}e^1\wedge e^{-2} + p_{1,-1}e^1\wedge e^{-1} + {} \\
& + p_{2,-2}e^2\wedge e^{-2} + p_{2,-1}e^2\wedge e^{-1} + p_{-2,-1}e^{-2}\wedge e^{-1} \longmapsto \\
& \longmapsto p_{1,2}e^1 + p_{1,-2}e^2 + p_{2,-1}e^{-2} + p_{-2,-1}e^{-1}.
\end{alignat*}
In other words, with respect to the basises $e^1$, $e^2$, $e^{-2}$, $e^{-1}$ and $e^1\wedge e^2$, $e^1\wedge e^{-2}$, $e^1\wedge e^{-1}$, $e^2\wedge e^{-2}$, $e^2\wedge e^{-1}$, $e^{-2}\wedge e^{-1}$ it is given by the matrix
\[ \begin{pmatrix}
1&0&0&0&0&0 \\
0&1&0&0&0&0 \\
0&0&0&0&1&0 \\
0&0&0&0&0&1
\end{pmatrix}\]
To avoid confusion, the mappings induced by $\tau$ on $V$ and $\wedge^2V$ will be denoted by $\tau_4$ and $\tau_6$ respectively. Obviously, $\rho\tau_6=\tau_4\rho$.


Symplectic group $Sp(4, R)$ is the subgroup of $GL(4, R)$ of all matrices $g=(g_{ij})$ such that $g^tsg=s$, where $s=\operatorname{antidiag}(1,1,-1,-1)$ is the Gram matrix of a symplectic bilinear form. $Sp(4, R)$ acts naturally on $V$.

Denote by $\mathcal{V}$ the submodule of $\wedge^2V$ of all vectors $v$ having $v_{1,-1}=v_{2,-2}$, i.e. $\mathcal{V}=\langle e^1\wedge e^2, e^1\wedge e^{-2}, e^2\wedge e^{-1}, e^{-2}\wedge e^{-1}, v^0 \rangle$, where $v^0=e^1\wedge e^{-1}+e^2\wedge e^{-2}$. This submodule is $Sp(4, R)$-invariant, and the action is the $5$-dimensional ``short roots'' representation (equivalently, it is the natural representation of $SO(5, R)$).

We define Suzuki group as
\[ Sz(R, \tau) = \{ g\in Sp(4, R) \mid \forall v\in \mathcal{V}\ \ \rho(g\wedge g)\tau_6(v) = \tau_4 g\rho(v)  \}. \]

\begin{lemma}\label{lemma:SzGroup}
$Sz(R, \tau)$ is a group.
\end{lemma}
\begin{proof}
Let $f, g\in Sz(R)$. One has to prove that $fg$ and $g^{-1}$ satisfy the defining relation. To show that two $\tau$-semilinear operators coincide, it is sufficient to check that their actions on the generators agree.

We will need a linear map $\hat\rho\colon V\to\wedge^2V$ given by the transpose of the matrix of $\rho$. One has $\rho\hat\rho=\operatorname{id}_V$.

Let us first show that $Sz(R, \tau)$ is closed under multiplication.

Rewrite $g\rho(e^1\wedge e^2)$ as $g_{*1}=g_{11}e^1+g_{21}e^2+g_{-2,1}e^{-2}+g_{-1,1}e^{-1}$, the first column of $g$. Note that it is equal to
$\rho(\hat\rho(g_{*1}))$ and that $\hat\rho(g_{*1})\in \mathcal{V}$. Then the left-hand side of the equation for $fg$ gives
\[
\tau_4 fg\rho(e^1\wedge e^2) = \tau_4f\rho\hat\rho(g_{*1}) = \rho(f\wedge f)\tau_6(\hat\rho(g_{*1}))
\]
The matrix expression of the exterior square gives
\[
(g\wedge g)\tau_6(e^1\wedge e^2) = \sum_{i<j} (g_{i1}g_{j2}+g_{j1}g_{i2}) e^i\wedge e^j
\]
On the other hand, the assumption $g\in Sz(R, \tau)$, applied to $v=e^1\wedge e^2$, translates into the following equations:
\begin{align*}
& g_{11}g_{22}+g_{12}g_{21}=g_{11}^\tau, \\
& g_{11}g_{-2,2}+g_{12}g_{-2,1} = g_{21}^\tau, \\
& g_{21}g_{-2,2}+g_{22}g_{-1,1} = g_{-2,1}^\tau, \\
& g_{-2,1}g_{-1,2}+g_{-2,2}g_{-1,1} = a_{-1,1}^\tau.
\end{align*}
Moreover, the examination of the entries in the upper-left corner of $g^tsg$ and $s$ shows that $g_{11}g_{-1,2}+g_{12}g_{-1,1}+g_{21}g_{-2,2}+g_{22}g_{-2,1}=0$. Together these five equations mean that
\[
(g\wedge g)\tau_6(e^1\wedge e^2) =
g_{11}^te^1\wedge e^2 + g_{21}^\tau e^1\wedge e^{-2} + g_{-2,1}^\tau e^2\wedge e^{-1} + g_{-1,1}^\tau e^{-2}\wedge e^{-1}
 + y\cdot v^0,
\]
where $y=g_{11}g_{-1,2}+g_{12}g_{-1,1}=g_{21}g_{-2,2}+g_{22}g_{-2,1}$. In other words,
\[
\tau_6\hat\rho(g_{*1})=(g\wedge g)\tau_6(e^1\wedge e^2) + y\cdot v^0.
\]
Since $f\in Sz(R, \tau)$ and $\rho(v^0)=0$, one has
\begin{gather*}
\rho(f\wedge f)(yv^0) = \rho(f\wedge f)\tau_6(yu) = \tau_4f\rho(yv^0) = 0. \\
\shortintertext{Thus}
\rho(f\wedge f)\tau_6(\hat\rho(g_{*1})) = \rho(f\wedge f) \big( (g\wedge g)\tau_6(e^1\wedge e^2) + yv^0 \big) = \rho(fg\wedge fg)\tau_6(e^1\wedge e^2).
\end{gather*}

The same holds for other basis vectors $e^i\wedge e^j$.

Consider now the action on $v^0$. The right-hand side evaluates to $0$. The assumption $g\in Sz(R, \tau)$ implies $\rho(g\wedge g)\tau_6(v^0)=0$, in particular, $g\wedge g(v^0)=zv^0$ for some $z\in R$, and the same holds for all scalar multiples of $v^0$. Hence
\[ \rho(fg\wedge fg)\tau_6(v^0)=\rho(f\wedge f)(zv^0)=\rho(z^2v^0)=0. \]

To show that $Sz(R, \tau)$ is closed under taking inverses note that for $g\in Sp(4, R)$ the inverse can be expressed as $g^{-1}=sg^ts$. A trivial check shows that $s\in Sz(R, \tau)$. Thus it is sufficient to show that $Sz(R, \tau)$ is closed under taking transposes.

Consider
\[ (g^t\wedge g^t)\tau_6(e^1\wedge e^2) = \sum_{i<j} (g_{1i}g_{2j}+g_{1j}g_{2i}) e^i\wedge e^j. \]
The equation $\rho(g\wedge g)\tau_6(e^i\wedge e^j)=\tau_4 g\rho(e^i\wedge e^j)$ for $i+j\neq0$ implies, by comparison of coefficients of $e^1=\rho(e^1\wedge e^2)$, that $g_{1i}g_{2j}+g_{1j}g_{2i}=g_{k1}^\tau$, where $k$ is such that $\hat\rho e^k=e^i\wedge e^j$. This shows that $(g^t\wedge g^t)\tau_6(e^1\wedge e^2)$ and $\tau_6\hat\rho((g^t)_{*1})$ differ by a multiple of $v^0$, hence $\rho(g^t\wedge g^t)\tau_6(e^1\wedge e^2)$ coincides with $\rho\tau_6\hat\rho((g^t)_{*1})=\tau_4\rho\hat\rho(g^te^1)=\tau_4 g^t\rho(e^1\wedge e^2)$.

Similar considerations show that the equality holds for the images of $e^i\wedge e^j$ with $i+j\neq0$.

As for the action on $v^0$, the right-hand side is again $0$, while for the left-hand side
\[
(g^t\wedge g^t)\tau_6(v^0) = \sum_{i<j} (g_{1i}g_{-1,j}+g_{1j}g_{-1,i}+g_{2i}g_{-2,j}+g_{2j}g_{-2,i}) e^i\wedge e^j.
\]
The coefficient of $e^i\wedge e^j$ in the above expression appears as the $(i, j)$-entry of $g^tsg$. Since $g\in Sp(4, R)$, one has $(g^t\wedge g^t)\tau_6(v^0)=v^0$ and hence $\rho(g^t\wedge g^t)\tau_6(v^0)=0$.
\end{proof}
Let us now define the analogue of the elementary subgroup in the context of Suzuki group. An upper unitriangular matrix that lies in $Sz(R, \tau)$, is one of the form
\[
x_+(a, b) = x_\alpha(a)x_\beta(a^\tau)x_{\alpha+\beta}(b)x_{2\alpha+\beta}(a^{2+\tau}+b^\tau) =
\begin{pmatrix}
1 & a & b+a^{1+\tau} & a^{2+\tau} +b^\tau+ab \\
& 1 & a^\tau & b \\
&& 1 & a \\
&&& 1
\end{pmatrix}.
\]
Such elements are subject to the relation
\[ x_+(a, b)\cdot x_+(c, d) = x_+(a+c, b+d+a^\tau c). \]
We also define $x_-(a, b)=s^{-1}x_+(a,b)s$. Define
\[ U=U^+ = \{ x_+(a, b) \mid a, b\in R \},\quad U^-={}^sU= \{ x_-(a, b) \mid a, b\in R \}. \]

The diagonal matrices that lie in $Sz(R, \tau)$ are of the form
\[ h(t) = h_\alpha(t)h_\beta(t^\tau) = \operatorname{diag}(t, t^{\tau-1}, t^{1-\tau}, t^{-1}). \]
A straightforward computation shows that
\begin{gather}
s\cdot h(t) = x_+(t^{1-\tau},t^{-1})\cdot x_-(0,t)\cdot x_+(t^{1-\tau}, 0) = x_-(0,t)^{x_+(t^{1-\tau}, 0)}, \label{eq:weyl-elements} \\
x_+(a,b)^{h(t)} = x_+(t^{\tau-2}a, t^{-\tau}b). \nonumber \\
\shortintertext{Denote}
H= \{ h(t) \mid t\in R^* \},\quad W=\{ 1, s \}\cong C_2,\quad B=HU,\quad N=WH. \nonumber
\end{gather}

We will now turn our attention to the Suzuki groups over a field (not necessarily perfect). First we will prove that $Sz(F, \tau)$ admits Bruhat decomposition. Since $\tau$ is not invertible, we do not have an endomorphism of $Sp(4, F)$ that would allow to carry the decomposition from the symplectic group as in \cite[Proposition~13.5.3]{CarterLie}. Neither we know whether $Sz(F, \tau)$ is generated by $U^+$ and $U^-$ to repeat the proof in \cite[Theorems~4 and~33(c)]{SteinbergChevalley}. Finally, we have nothing prepared for the geometric considerations as in \cite[n\textsuperscript{o}\,4]{TitsSuzukiRee}. Thus for the lack of ingenuity we will stick with the brute force.

\begin{lemma}\label{lemma:Bruhat}
$Sz(F, \tau) = UWHU$.
\end{lemma}
\begin{proof}
Fix $g\in Sz(F, \tau)$. It is sufficient to prove that for suitable $u\in U$ one has $ug\in WHU$.

Assume first that $g_{-1,1}=0$. Then the defining relation $\rho(g\wedge g)\tau_6(v) = \tau_4 g\rho(v)$, evaluated at $v=e^1\wedge e^2$, gives $g_{-2,1}g_{-1,2}\cdot e^{-1}=0$. If $g_{-2,1}=0$, then the evaluation at $v=e^1\wedge e^{-2}$ gives $g_{-1,2}^\tau\cdot e^{-1}=0$. If $g_{-1,2}=0$, then the evaluation at $v=e^1\wedge e^2$ gives $g_{-2,1}\tau\cdot e^{-2}=0$. In either case, both $g_{-2,1}$ and $g_{-1,2}$ equal $0$. Now substitute $v=v^0$ to show that $g_{-2,2}g_{-1,-2}\cdot e^{-1}=0$. Substituting $v=e^1\wedge e^2$ and $e^1\wedge e^3$ and looking at the coefficients of $e^2$ and $e^{-2}$ respectively shows that $g_{-2,2}=g_{-1,-2}=0$. Then also $g_{21}=0$, and so $g$ is upper-triangular. It is easy to see that $g=h(t)x_+(a,b)$ for suitable $t$, $a$ and $b$. Namely, take $t=g_{11}$, $a=g_{11}^{-1}g_{12}$ and $b=g_{11}^{1-\tau}g_{2,-1}$. In this case $u=1$.

Consider now the case $g_{-1,1}\neq 0$. Set $u=x_+\big(g_{-1,1}^{-1}g_{-2,1}, g_{-1,1}^{-1}(g_{21}+g_{-2,1}^{\tau+1}g_{-1,1}^{-\tau})\big)$. Set $f=ug$, then $f_{21}=f_{-2,1}=0$ and $f_{-1,1}=g_{-1,1}\neq0$. Then the evaluation of the defining relation for $f$ at $v=e^1\wedge e^2$ shows that $f_{22}f_{-1,1}=0$ and $f_{11}f_{22}=f_{11}^\tau$, so $f_{22}=0$ and hence $f_{11}=0$. Now substituting $v=e^1\wedge e^2$ and $e^2\wedge e^{-1}$ gives $f_{12}^\tau=0$ and $f_{12}f_{2,-1}=f_{1,-2}^\tau$, so $f_{12}=f_{1,-2}=0$ and $f$ is of the form $s\,h(t)x_+(a,b)$, where $t=f_{-1,1}$, $a=f_{-1,1}^{-1}f_{-1,2}$, $b=f_{-1,1}^{1-\tau}f_{-2,-1}$.
\end{proof}

The following is a special case of Tits' simplicity theorem~\cite[Theorem~11.1.1]{CarterLie}
\begin{proposition*}
If $G$ is a perfect group with BN-pair of rank $1$ such that $B$ is solvable and core-free, then $G$ is simple.
\end{proposition*}

\begin{lemma}\label{lemma:SzSimple}
If $F\neq\mathbb{F}_2$, the group $Sz(F, \tau)$ is simple.
\end{lemma}
\begin{proof}
It is easy to see that $B\cap s^{-1}Bs=H$. Consider now the subgroup $B^{x_-(0, 1)}$. Its elements are of the form $g=x_-(0,1)h(t)x_+(a, b)x_-(0,1)$. A straightforward calculations show that $g_{2,-2}=t^{\tau-1}a^\tau$ and $g_{2,-1}=t^{\tau-1}b$. If $g\in H$, then $a=b=0$. This implies $g_{-1,1}=t+t^{-1}$, so $g\notin H$ unless $t=1$. In the latter case $g=1$. Hence $\bigcap_{g\in G}B^g\subseteq B\cap B^s\cap B^{x_-(0, 1)}=1$, so $B$ is core-free.

To prove that $G$ is perfect, it suffices to present its generators as products of commutators. First note that $x_+(a, b)=x_+(a, 0)x_+(0, b)$ and that by \cref{lemma:Bruhat} and \cref{eq:weyl-elements} the subgroups $U^+$ and $U^-$ generate $G$. Now consider
\[ [x_+(0, b), h(t)] = x_+(0, b(1+t^\tau)). \]
Since $F\neq\mathbb{F}_2$, one can take $t\neq1$ and set $b=d/(1+t^\tau)$. Then this commutator equals $x_+(0, d)$. On the other hand,
\[ [x_+(a, 0), h(t)] = x_+\big(a(1+t^{2-\tau}), a^{1+\tau}(t^\tau+t^{2-\tau})\big). \]
Since $\tau$ in injective and $\tau^2$ is the Frobenius endomorphism, one has
\[ 1+t^{2-\tau}=0\ \Rightarrow\ t^\tau=t^2\ \Rightarrow\ t=t^\tau\ \Rightarrow\ t=t^2\ \Rightarrow t=1. \]
So taking $t\neq1$ and $a=c/(1+t^{2-\tau})$ gives $[x_+(a, 0), h(t)] = x_+(c, 0)x_+(0, *)$. Multiplication by a suitable $x_+(0, *)$, which is a commutator, delivers the expression of $x_+(c, d)$ as a product of two commutators. An element of the form $x_-(c, d)$ has a similar expression since it is conjugate to $x_+(c, d)$.

Checking that $(B, N)$ form a rank $1$ $BN$-pair for $G$ is easy. Thus by the proposition above $Sz(F, \tau)$ is simple.
\end{proof}
\bibliographystyle{amsalpha}
\bibliography{suzuki}
\end{document}